\documentclass[12pt]{article}
\usepackage[english]{babel}
\usepackage{graphicx}
\usepackage{amsmath}
\usepackage{amsfonts}
\usepackage{amssymb}
\usepackage{amsthm}
\usepackage{hyperref}
\usepackage[T1]{fontenc}
\usepackage{color}
\usepackage{mathtools}
\usepackage{comment}

\input xy
\xyoption{all}

\newtheorem{thm}{Theorem}
\newtheorem{cor}[thm]{Corollary}
\newtheorem{lem}[thm]{Lemma}
\newtheorem{pro}[thm]{Proposition}

\newtheorem*{Mthm}{Main Theorem}

\theoremstyle{definition}
\newtheorem{defi}[thm]{Definition}

\newtheoremstyle{remarque}{}{}{}{}{\it}{.}{\newline}{}
\theoremstyle{remarque}
\newtheorem*{rem}{Remark}

\newcommand{\asd}[5]{% 
\setbox1=\hbox{\ensuremath{^{#1}}}% 
\setbox2=\hbox{\ensuremath{_{#2}}}% 
\setbox5=\hbox{\ensuremath{#5}}% 
\hspace{\ifnum\wd1>\wd2\wd1\else\wd2\fi}% 
\ensuremath{\copy5^{\hspace{-\wd1}\hspace{-\wd5}#1\hspace{\wd5}#3}% 
_{\hspace{-\wd2}\hspace{-\wd5}#2\hspace{\wd5}#4}% 
}}

\usepackage[OT2,T1]{fontenc}
\DeclareSymbolFont{cyrletters}{OT2}{wncyr}{m}{n}
\DeclareMathSymbol{\Sha}{\mathalpha}{cyrletters}{"58}
\DeclareMathSymbol{\Brusse}{\mathalpha}{cyrletters}{"42}

\newcommand{\n}{\mathbb{N}}
\newcommand{\z}{\mathbb{Z}}
\newcommand{\q}{\mathbb{Q}}
\newcommand{\C}{\mathbb{C}}

\renewcommand{\int}{\mathrm{Int}\,}
\newcommand{\gm}{\mathbb{G}_{\mathrm{m}}}
\newcommand{\br}{\mathrm{Br}\,}
\newcommand{\nr}{\mathrm{nr}}

\newcommand{\brnr}{{\mathrm{Br}_{\nr}}}

\newcommand{\gal}{\mathrm{Gal}}
\renewcommand{\sl}{\mathrm{SL}}
\newcommand{\sln}{\mathrm{SL}_n}

\newcommand{\pic}{\mathrm{Pic}\,}

\newcommand{\tor}{\mathrm{tor}}
\renewcommand{\ss}{\mathrm{ss}}
\newcommand{\ssu}{\mathrm{ssu}}
\newcommand{\red}{\mathrm{red}}
\newcommand{\torf}{{\rm torf}}
\newcommand{\f}{\mathrm{f}}
\newcommand{\un}{\mathrm{u}}
\newcommand{\spec}{\mathrm{Spec}\,}

\newcommand{\bb}[1]{\mathbb{#1}}

\usepackage{marvosym}

\title{Weak approximation for homogeneous spaces: reduction to the case with finite stabilizer}
\author{Giancarlo Lucchini Arteche\\[5mm]
{\it\small Centre de Math\'ematiques Laurent Schwartz, UMR CNRS 7640}\\
{\it\small \'Ecole polytechnique, 91128 Palaiseau, France}\\
{\small giancarlo.lucchini-arteche@polytechnique.edu }
}

\date{}

\begin{document}

\maketitle

\begin{abstract}
We reduce the question about whether the Brauer-Manin obstruction to weak approximation for homogeneous spaces is the only obstruction to the ``simpler'' question of the particular case of homogeneous spaces of $\sln$ with finite stabilizer. We also give unconditional results based on the few known results on finite stabilizers.\\

{\bf MSC classes (2010):} 14M17, 14G05.
\end{abstract}

\section{Introduction}
Let $k$ be a number field and let $V$ be a smooth, geometrically integral $k$-variety. Recall that $V$ is said to have \emph{weak approximation} if its $k$-points are dense in the product of its $k_v$-points, where $k_v$ represents the different completions of $k$. The question of whether $V$ has or not weak approximation is, together with the respective question on the Hasse principle, a classic one in arithmetic geometry. Note however that the question on weak approximation is of interest only if $V$ already has a $k$-point, otherwise we would be interested first in the validity of the Hasse principle, which would settle the question of the existence of $k$-points.\\

Let now $G$ be a connected linear algebraic $k$-group. Recall that a ($k$-)homogeneous space of $G$ is a variety $V$ equipped with a $k$-action of $G$ which is transitive at the level of geometric points. In the particular case where $V$ has a $k$-point, then it is easier to say that it is isomorphic to the quotient $G/H$ for some $k$-subgroup $H$ of $G$, which can be taken to be the stabilizer of the given $k$-point for the given action. The question of weak approximation for such homogeneous spaces has been studied since at least half a century. In particular, if we look at homogeneous spaces with trivial stabilizer, we come to the question of whether an algebraic group $G$ itself has weak approximation as a $k$-variety. We know today that the answer to this question is positive for unipotent or semisimple simply connected groups, while it may be false for tori or for other semisimple groups. These results, due to Kneser, Voskresenskii and others, are subsumed in a very clear way in \cite{Sansuc81}, the 
conclusion being that the \emph{Brauer-Manin obstruction} is the only obstruction to weak approximation in this case.\\

If one considers now the same question for homogeneous spaces with non-trivial stabilizers, a lot of work was done after Sansuc's article, ending in Borovoi's result in \cite{Borovoi96} stating that the Brauer-Manin obstruction is the only obstruction to weak approximation if we assume the stabilizer $H$ to be \emph{connected}. However, taking away this hypothesis seems to be a task of a much higher level of difficulty\footnote{In Borovoi's result, $H$ can be taken to be non-connected if it is abelian and if one adds a technical hypothesis on $G$.}. Already the ``simplest'' case of a non-connected stabilizer, that is, the case of \emph{finite} stabilizers, poses a lot of problems and only some very partial results have been found (as for example in \cite{HarariBulletinSMF}, \cite{Demarche} and \cite{GLA AF PH res}).

The goal of this article is to prove that this ``simple'' case actually accounts for the whole question. This makes sense since it is well known that every algebraic group is an extension of a finite group by a connected group. But note however that, as Colliot-Th\'el\`ene has noticed, if the Brauer-Manin obstruction happens to be the only obstruction to weak approximation for homogeneous spaces with finite stabilizer, then an easy corollary is the positive answer to the inverse Galois problem\footnote{This goes all the way back to a text by Ekedahl (cf. \cite{Ekedahl}) involving varieties verifying weak or strong approximation. Colliot-Th\'el\`ene noticed then that ``weak weak'' approximation would suffice to do the same arguments, and this last one is implied by Brauer-Manin.} (!), hence the quotation marks on ``simple''.\\

In this article then, we stick together an eventual result on finite stabilizers with Borovoi's result on connected stabilizers, much in the same way than Borovoi did when he put together in \cite{Borovoi96} his results on semisimple groups (cf. \cite{Borovoi93}) with the work done before him on tori and unipotent groups. More precisely, we prove:

\begin{Mthm}[Consequence of Theorem \ref{MainThm}]
Assume that the Brauer-Manin obstruction is the only obstruction to weak approximation for homogeneous spaces of the form $\sln/H$ with finite $H$. Then this is also the case for every homogeneous space $G/H$ with connected $G$ and arbitrary $H$.
\end{Mthm}

The actual theorem states that we may associate a precise finite $k$-group embedded into $\sln$ to each given homogeneous space $G/H$.\\

Concerning finite stabilizers, the case where $G=\sln$ is basically the only one that has been studied until now, certainly because one could intuit that everything would come down to this case, and it is in this context that links with Galois cohomology and the inverse Galois problem appear (as a matter of fact, inverse Galois is a consequence of an even weaker type of approximation property for these varieties, cf. \cite[\S4]{HarariBulletinSMF}). Moreover, by classic arguments that will reappear in the proof of the main theorem (cf. Lemma \ref{lemme stable bir} below), we know that all homogeneous spaces of the form $\sln/H$ for fixed $H$ are $k$-stably birational to each other and, by the No-name Lemma (cf. \cite[Cor. 3.9]{ColliotSansucChili}), we know that they are at the same time $k$-stably birational to the quotient $\mathbb{A}_k^n/H$ for any faithfull linear $k$-action of $H$ on affine space. This links the study of homogeneous spaces with Noether's problem, which asks about the $k$-rationality of such quotients. The failure of weak approximation is in this sense a proof of non-$k$-rationality.\\

The structure of the text is the following. In section \ref{preliminaires} we fix notations and we recall some properties of the Brauer-Manin obstruction. We also give some results on algebraic groups we will be using. In section \ref{section Ono} we revisit Ono's Lemma (a useful tool for producing quasi-split tori, cf. \cite[Thm 1.5.1]{Ono}) in the context of algebraic tori equipped with a $k$-action of a smooth finite $k$-group. We then go straight to the proof of the Main Theorem in section \ref{Proof}. The main idea is to use the same trick Borovoi used in \cite{Borovoi96} to eliminate the semisimple and the unipotent parts of the stabilizer, which will reduce us to the case where the stabilizer is an extension of a finite group by a torus, then apply our new version of Ono's Lemma in order to reduce us to the case where this torus is nice enough and finally apply the existence result established in \cite[Prop. 1]{GLABrnral2} in order to create a homogeneous space with finite stabilizer which will be $k$-stably birational to the one we have reduced to. We finally give in section \ref{section resultat inconditionnel} an unconditional result by putting together our Main Theorem with a result of Harari on finite stabilizers.

\paragraph*{Acknowledgements}
This work was partially supported by the FMJH through the grant N\textsuperscript{o} ANR-10-CAMP-0151-02 in the ``Programme des Investissements d'Avenir''.

\section{Notations and preliminaries}\label{preliminaires}
In this section, as well as in sections \ref{Proof} and \ref{section resultat inconditionnel}, $k$ denotes a number field with algebraic closure $\bar k$ and absolute Galois group $\Gamma_k$. The set of places of $k$ is denoted by $\Omega$ and $k_v$ denotes the completion of $k$ with respect to $v\in\Omega$. We use the letters $U,V,W$ for $k$-varieties. These will always be assumed to be smooth and geometrically integral since this is the case for homogeneous spaces. We denote by $V(k_\Omega)$ the product $\prod_\Omega V(k_v)$.

By Brauer group of $V$ we mean its cohomological definition $\br V:=H^2_{\text{\'et}}(V,\gm)$. The unramified Brauer group $\brnr V$ is defined as the Brauer group $\br X$ of a smooth compactification $X$ of $V$, that is, a proper smooth $k$-variety equipped with an open immersion $V\hookrightarrow X$ ($X$ always exists thanks to Nagata's and Hironaka's Theorems). For a birationally-invariant definition of this group and the proof of its equivalence with the one given here, we send the reader to \cite{ColliotSantaBarbara}.

\subsection{The Brauer-Manin obstruction}\label{section BM}
For the definition of the Brauer-Manin obstruction, we send the reader to \cite[\S5.2]{Skor}. We simply recall that there is a pairing
\begin{align*}
V(k_\Omega)\times\brnr V &\to \q/\z\\
((x_v),\alpha)&\mapsto \sum_{v\in\Omega}\mathrm{inv}_v(\alpha(x_v)),
\end{align*}
and that the set $V(k_\Omega)^\br$ is defined as the left kernel of this pairing, that is, the subset of $V(k_\Omega)$ of those elements which are orthogonal to the whole group $\brnr V$ for the paring above. By global class field theory, we have the inclusion $V(k)\subset V(k_\Omega)^\br$. Moreover, this last subset is closed in $V(k_\Omega)$, hence the closure $\overline{V(k)}$ of $V(k)$ in $V(k_\Omega)$ is also contained in $V(k_\Omega)^\br$.

\begin{defi}
We say that the Brauer-Manin obstruction to weak approximation is the only obstruction for $V$ if we have the equality $\overline{V(k)}=V(k_\Omega)^\br$. We abbreviate this by saying that $V$ ``satisfies BM-WA''.
\end{defi}

Note that the BM-WA property is invariant by $k$-stable birationality as soon as the quotient $\brnr V/\br k$ is finite. That is, if $V$ satisfies BM-WA and is $k$-stably birational to $W$, then $W$ also satisfies BM-WA. Recall that $V$ and $W$ are $k$-stably birational if there exists $r,s\in \n$ such that $V\times_k\bb{A}_k^r$ is $k$-birational to $W\times_k\bb{A}_k^s$. This invariance follows from \cite[Prop. 6.1]{ColliotPalSkor} ($k$-birational invariance of BM-WA), \cite[Thm. 4.1.5]{ColliotSantaBarbara} (homotopy invariance of $\brnr V$) and from the well-known fact that affine space has weak approximation.\\

Note that the so-called Brauer set $V(k_\Omega)^\br$ has the following properties:

\begin{pro}\label{prop BM}
Let $f:W\to V$ be a $k$-morphism of (smooth, geometrically integral) varieties. Then,
\begin{itemize}
 \item[(i)] the natural morphism $f_*:W(k_\Omega)\to V(k_\Omega)$ sends $W(k_\Omega)^\br$ to $V(k_\Omega)^\br$;
 \item[(ii)] if moreover $f$ admits a $k$-section $g:V\to W$, then for every $(x_v)\in V(k_\Omega)^\br$ its preimage $(y_v):=g_*((x_v))\in W(k_\Omega)$ belongs to $W(k_\Omega)^\br$.
\end{itemize}
\end{pro}

\begin{proof}
This is an easy consequence of the functoriality of the Brauer group. Indeed, for every $(x_v)\in W(k_\Omega)$ with image $(y_v)\in V(k_\Omega)$ and $\alpha\in\brnr V$ we have
\[\alpha(y_v)=f^*(\alpha)(x_v)\in\br k_v,\]
simply because $y_v$ corresponds by definition, for every $v\in\Omega$, to the composition
\[\spec k_v\xrightarrow{x_v} W\xrightarrow{f} V.\]
This gives immediately assertion (i). Assertion (ii) follows from (i) applied to $g$.
\end{proof}

\subsection{Homogeneous spaces}\label{section HS}
Let us talk now about homogeneous spaces. For $G$ a linear algebraic $k$-group, we denote
\begin{itemize}
\item $D(G)$ the derived subgroup of $G$;
\item $G^\circ$ the neutral connected compontent of $G$;
\item $G^\f=G/G^\circ$ the group of connected components of $G$ (it is a finite group);
\item $G^\un$ the unipotent radical of $G^\circ$ (it is a unipotent characteristic subgroup);
\item $G^\red=G^\circ/G^\un$ (it is a reductive group);
\item $G^\ss=D(G^\red)$ (it is a semisimple group);
\item $G^\tor=G^\red/G^\ss$ (it is a torus);
\item $G^\ssu=\ker[G^\red\twoheadrightarrow G^\tor]$ (it is an extension of $G^\ss$ by $G^u$);
\item $G^\torf=G/G^\ssu$ (it is an extension of $G^\f$ by $G^\tor$);
\end{itemize}
all of which are defined over $k$. For $T$ a $k$-torus, we shall note $T[m]$ its subgroup of $m$-torsion elements. Recall that $T$ is said to be \emph{quasi-split} if its geometric characters form an induced $\gal(\bar k/k)$-module.\\

As it was recalled in the introduction, we are only concerned here with homogeneous spaces of the form $V=G/H$ with $G$ a connected linear $k$-group and $H$ an arbitrary $k$-subgroup. In any case, recall that when the base field is of characteristic 0, one may define a homogeneous space of $G$ as a variety $V$ equipped with a $k$-action (i.e. a $k$-morphism $a:G\times_k V\to V$ having a natural compatibility with mutliplication in $G$) such that the induced action of $G(\bar k)$ over $V(\bar k)$ is transitive.\\

The proof of the Main Theorem uses some particular key results on algebraic groups which we recall here. First of all, there is the following result, consequence of \cite[Prop. 1]{GLABrnral2}, useful for producing finite stabilizers.

\begin{pro}\label{proposition existence}
Let $K$ be a field of characteristic 0. Let $G$ be a linear algebraic $K$-group such that $G=G^\torf$. Then there exists a finite $K$-subgroup $F$ of $G$ and a commutative diagram with exact lines
\[\xymatrix{
1 \ar[r] & G^\tor[m] \ar[r] \ar@{^{(}->}[d] & F \ar[r] \ar@{^{(}->}[d] & G^\f \ar[r] \ar@{=}[d] & 1 \\
1 \ar[r] & G^\tor \ar[r] & G \ar[r] & G^\f \ar[r] & 1,
}\]
where $m=nd$, $n$ is the order of $G^\f$ and $d$ the order of an extension $K'/K$ splitting $G^\tor$.
\end{pro}

Next, let us recall two lemmas used by Borovoi in \cite{Borovoi96}. They allow him to treat separately the case where the stabilizer is of multiplicative type and the case where it only has a semisimple and a unipotent part (cf. \cite[Lem. 3.1, Lem. 5.1]{Borovoi96}).

\begin{lem}\label{lemme fibration}
Let $K$ be any field. Let $V$ be a homogeneous space of an affine $K$-group $G$, and $N$ a normal $K$-subgroup of $G$. Then there exists a quotient $W = V/N$, i.e. a homogeneous space $W$ of $G/N$ and a smooth $G$-equivariant map $\varphi: V \to W$ such that $\varphi$ is a quotient map. In particular, $\varphi$ is surjective and its geometric fibers are orbits of $N$.
\end{lem}

\begin{lem}\label{lemme reduction sc}
Let $G$ be a connected $k$-group. Then there exists an extension
\[1 \to T \to G' \to G \to 1,\]
such that $T$ is a torus and $(G')^\ssu$ is simply connected.
\end{lem}

We will use the first lemma in the same fashion: we will separate the unipotent and semisimple parts (i.e. $H^\ssu$) from the toric and finite parts (i.e. $H^\torf$), which will be dealt with using Proposition \ref{proposition existence} and the results of section \ref{section Ono}. The ``ssu'' part is dealt with using the second lemma and the fact (cf. \cite[Pro. 3.4]{Borovoi96}) that homogeneous spaces with simply connected ambient space and stabilizer of ``ssu'' type have weak approximation.\\

The argument using Lemma \ref{lemme fibration} appears however more than once during the proof of the Main Theorem. We finish then this section by giving a precise statement in order to avoid repetitions later.\footnote{I warmly thank the anonymous referee for suggesting this.}

\begin{lem}\label{lemme stable bir}
Let $K$ be any field, let $G$ be a connected $K$-group and let $H$ be a $K$-subgroup of $G$. For $i=1,2$, let $G_i$ be a $K$-rational special reductive group and let $\iota_i:H\to G_i$ be a $K$-group morphism inducing a diagonal embedding $H\to G\times G_i$. Then the homogeneous spaces $(G\times G_1)/H$ and $(G\times G_2)/H$ are $K$-stably birational.
\end{lem}

\begin{proof}
It will suffice to prove that both homogeneous spaces are $K$-stably birational to $(G\times G_1\times G_2)/H$. By symmetry, we are then reduced to the case where $G_2$ is trivial (replace $G$ by $G\times G_2$ here above).

Put then $G_0:=G\times G_1$, $V=G_0/H$ and $W=G/H$. Applying Lemma \ref{lemme fibration} to $G_0$ and its normal subgroup $G_1$, we get a surjective map $V\to W$ whose fibers are orbits of $G_1$ which are easily seen to be torsors under this group. Since $G_1$ is special, the fiber over any point $x\in V$ is then isomorphic to $G_1$ over the residue field of $x$. This is the case in particular for the generic fiber and hence, since $G_1$ is $K$-rational, we get that $W$ is rational over $K(V)$, the function field of $V$. In other words, $W$ is $K$-stably birational to $V$.
\end{proof}

\section{Ono's Lemma revisited}\label{section Ono}
We give here a new interpretation of Ono's Lemma (cf. \cite[Thm. 1.5.1]{Ono}). This result allows us to produce quasi-split tori, which in turn are useful for proving $K$-stable birationality between homogeneous spaces.\footnote{I warmly thank Jean-Louis Colliot-Th\'el\`ene for bringing this up in an old conversation.}

Let us first recall the original statement. A morphism of tori $T_0\to T_1$ is said to be an \emph{isogeny} if it is surjective and with finite kernel. Then Ono's Lemma states:

\begin{lem}\label{lemme Ono}
Let $K$ be any field. Let $T$ be a $K$-torus and $K'/K$ a Galois extension splitting $T$. Then there exists an integer $r\geq 1$ and a quasi-split torus $T_0$, split by $K'/K$, such that $T^r\times T_0$ is $K$-isogenous to a quasi-split torus.
\end{lem}

Now, Ono's proof works at the level of geometric characters. Indeed, the assertion of his lemma is easily seen to be equivalent to the following assertion:

\begin{lem}
Let $K$ be any field, $K'/K$ a Galois extension, $\Gamma_{K'/K}$ the corresponding Galois group and $M$ a free $\Gamma_{K'/K}$-module of finite type. Then there exists an integer $r\geq 1$ and induced $\Gamma_{K'/K}$-modules $M_0,M_1$ such that $M_1$ is a sub-$\Gamma_{K'/K}$-module of finite index of $M^r\times M_0$.
\end{lem}

And since Ono's Lemma is valid for any field $K$, this means that the proof works for any (pro)finite group. In other words, if we ``copy-paste'' Ono's proof (which we do here below for the comfort of the reader) we get the following result:

\begin{lem}
Let $\Gamma$ be a finite (abstract) group and $M$ a free $\Gamma$-module of finite type. Then there exists an integer $r\geq 1$ and induced $\Gamma$-modules $M_0,M_1$ such that $M_1$ is a sub-$\Gamma$-module of finite index of $M^r\times M_0$.
\end{lem}

\begin{proof}
For any $\Gamma$-module $M$, we may consider the associated rational (resp. complex) representation $\rho_M:\Gamma\to\mathrm{GL}(V)$, where $V=M\otimes_\z\q$ (resp. $M\otimes_\z\C$), and we get then an associated character $\chi_M(\gamma):=\mathrm{Tr}(\rho_M(\gamma))$ for $\gamma\in\Gamma$. Characters obtained in this way are clearly \emph{rational}, that is, they take values in $\q$.

Let now $X$ be the $\z$-module generated by all rational characters of $\Gamma$. It is a sub-$\z$-module of finite type of the finite-dimensional vector space of maps $\Gamma\to\q$. Let then $\Delta_i$ ($1\leq i\leq k$) be representatives of the conjugation classes of all cyclic subgroups of $\Gamma$ and consider, for each one of them, the trivial character $\chi_i:\Delta_i\to\q$ associated to it. These give us non-trivial and rational induced characters $\chi_i^\Gamma$ of $\Gamma$ and Artin's Induction Theorem (cf. for example \cite[\S9]{SerreGpsFinis}) tells us then that the $\z$-module generated by the $\chi_i^\Gamma$'s is a submodule of $X$ of \emph{finite} index $r$.\\

Let now $M$ be a free $\Gamma$-module and let $\chi_M$ be its associated character. Since $\chi_M$ is rational, we know then that there exist integers $n_i,m_i\geq 0$ such that
\[r\chi_M+\sum_{i=1}^kn_i\chi_i^\Gamma=\sum_{i=1}^km_i\chi_i^\Gamma.\]
In particular, if we look at the associated (rational) representations of these characters, we get
\[(M\otimes_\z\q)^{r}\oplus\bigoplus_{i=1}^k\q[\Gamma/\Delta_i]^{n_i}\cong\bigoplus_{i=1}^k\q[\Gamma/\Delta_i]^{m_i}.\]
Finally, this $\q$-isomorphism induces, up to multiplication by an integer, a $\Gamma$-equivariant injection with finite cokernel
\[\bigoplus_{i=1}^k\z[\Gamma/\Delta_i]^{m_i}\to M^r\times\bigoplus_{i=1}^k\z[\Gamma/\Delta_i]^{n_i},\]
which amounts to say that the induced $\Gamma$-module on the left (which we may call $M_1$) is a sub-$\Gamma$-module of finite index of the product of $M^r$ and a second induced $\Gamma$-module (which we may call $M_0$).
\end{proof}

We are now able to generalize the use that Ono makes of this lemma to a slightly more general context. Let us concentrate then on $K$-groups $H$ of ``$\torf$'' type (that is, such that $H=H^\torf$, see notations in section \ref{section HS}) for an arbitrary field $K$. In order to lighten notation, we will note $T=H^\tor=H^\circ$, $F=H^\f=H/T$ and $\Gamma=\Gamma_K$ during the rest of this section. We know then that $F$ acts on the torus $T$ by conjugation. Moreover, if we denote by $F_{\Gamma}$ the abstract profinite group $F(\bar K)\rtimes\Gamma$ (the semi-direct product being taken with respect to the natural action of $\Gamma$ over $F(\bar K)$), then the $K$-action of $F$ on $T$ corresponds to an action of $F_\Gamma$ on $T(\bar K)$ whose restriction to $\Gamma$ (seen as a subgroup of $F_\Gamma$ by the natural section) is the natural Galois action (see \cite[\S3]{GLAExt} for details). By duality, we also obtain a continuous action of $F_\Gamma$ on the module $\hat T$ of geometric characters of $T$. It is easy to see then that, when $F$ is smooth, any such action defines a $K$-torus equipped with a $K$-action of $F$ by group automorphisms. Applying the result above in this context we get:

\begin{cor}
Let $K$ be any field, $F$ a smooth finite $K$-group and $M$ a free $F_\Gamma$-module of finite type. Then there exists an integer $r\geq 1$ and induced $F_\Gamma$-modules $M_0, M_1$ such that $M_1$ is a sub-$F_\Gamma$-module of finite index of $M^r\times M_0$.
\end{cor}

\begin{rem}
Note that we have actually proved that the induced modules $M_0, M_1$ are induced $(F(\bar K)\rtimes\Gamma_{K'/K})$-modules, where $K'/K$ is a Galois extension splitting both $F$ and $T$ (for $F$, this means that $F\times_K{K'}$ is a constant group). This last semi-direct product is well defined and corresponds to the quotient of $F_\Gamma$ by its normal subgroup $\Gamma_{K'}$.
\end{rem}

Let us now translate back what we have just proved to the language of tori.

\begin{defi}
Let $K$ be a field and $F$ a finite $K$-group. We define an \emph{$F$-torus} to be a $K$-torus equipped with a $K$-action of $F$ by group automorphisms.

An $F$-torus is said to be \emph{$F$-quasi-split} if its module of geometric characters $\hat T$ is an induced $F_\Gamma$-module.

An $F$-morphism of $F$-tori is an $F$-equivariant $K$-morphism of tori. An $F$-isogeny is then a surjective $F$-morphism with finite kernel.
\end{defi}

In particular, any algebraic group $H$ defines (via its quotient $H^\torf$) an $H^\f$-torus $H^\tor$, as we mentionned above.

\begin{cor}\label{lemme super Ono}
Let $K$ be any field, $F$ a smooth finite $K$-group, $T$ an $F$-torus and $K'/K$ a Galois extension splitting $T$ and $F$. Then there exists an integer $r\geq 1$ and an $F$-quasi-split torus $T_0$, split by $K'/K$, such that $T^r\times T_0$ is $F$-isogenous to an $F$-quasi-split torus.
\end{cor}

The notion of $F$-quasi-split torus will be essential in order to prove the Main Theorem. Its main feature in what concerns us is the following: whenever we are given an $F$-torus $T$, we may consider the twisted torus $\asd{}{X}{}{}{T}$ for any principal homogeneous space (or torsor) $X$ of $F$ which, let us recall, are classified up to isomorphism by the Galois cohomology set $H^1(K,F)$. The torus $\asd{}{X}{}{}{T}$ is always a $K$-torus, but in general not an $F$-torus (it has however the structure of a $\asd{}{X}{}{}{F}$-torus, but this does not concern us). Concerning quasi-split tori, one could hope that they would be stable under twisting, i.e. that the twist of a quasi-split torus is also quasi-split, but this is far from being true in general. However, if our torus happens to be \emph{$F$-quasi-split}:

\begin{pro}\label{prop tore tordu}
Let $K$ be any field, $F$ a smooth finite $K$-group and $T$ an $F$-quasi-split torus. Let $X$ be an $F$-torsor over $K$. Then the twisted torus $\asd{}{X}{}{}{T}$ is quasi-split as a $K$-torus.
\end{pro}

\begin{proof}
It is a classic result in group cohomology that the set $H^1(K,F)$ classifies, up to conjugation, the continuous sections of the split exact sequence
\[\xymatrix{
1 \ar[r] & F(\bar K) \ar[r] & F_\Gamma \ar[r] & \Gamma \ar[r] \ar@/_1pc/[l]_s & 1,
}\]
where $s$ denotes the natural section (cf. for example \cite[I.5.1, Exer. 1]{SerreCohGal}). In particular, a cocycle $x\in Z^1(K,F)$ given by $X$ explicitly defines such a section $s_x$ by sending $\sigma\in\Gamma$ to $s_x(\sigma):=x_\sigma s(\sigma)\in F_\Gamma$. Now, by the very definition of twisting, we see that the action of $\Gamma$ on $\asd{}{X}{}{}{T}(\bar K)=T(\bar K)$ is nothing but the restriction of the action of $F_\Gamma$ to $s_x(\Gamma)$ and it goes the same way for $\asd{}{X}{}{}{\hat T}=\hat T$. Then, since $T$ is $F$-quasi-split, $\hat T$ is an induced $F_\Gamma$-module and hence an induced $s_x(\Gamma)$-module by restriction, telling us that $\asd{}{X}{}{}{T}$ is quasi-split as a $K$-torus.
\end{proof}

\section{Proof of the Main Theorem}\label{Proof}
We restate here the Main Theorem in a more precise version, which we prove in what follows. Recall that $H^\f$ denotes the finite quotient group of connected components of $H$ (see notations in section \ref{section HS}).

\begin{thm}\label{MainThm}
Let $k$ be a number field, $G$ a connected algebraic $k$-group, $H$ an arbitrary $k$-subgroup of $G$ and let $V=G/H$. Then there exists a finite $k$-group $F$, extension of $H^\f$ by an abelian group, such that, if the Brauer-Manin obstruction to weak approximation is the only obstruction for $\sln/F$ (for some embedding of $F$ into $\sln$), then this is also the case for $V$.
\end{thm}

\begin{rem}
As we remarked in the Introduction, Lemma \ref{lemme stable bir} tells us that we do not need to care about the embedding of $F$ into $\sln$ since all such quotients are $k$-stably birational to each other. The BM-WA property (see notations in section \ref{section BM}) for $\sln/F$ depends only on $F$ then.
\end{rem}

\begin{proof}
We reduce first using Borovoi's results and Ono's Lemma to a homogeneous space with ambient space $G=\sln\times G^\tor$ and stabilizer $H^\torf$ (step 1). Next, using our improved version of Ono's Lemma, we reduce to the case where $H^\tor$ is isogenous to an $H^\f$-quasi-split torus (step 2). We treat then this particular case by constructing a homogeneous space which is $k$-stably birational to it but with finite stabilizer (step 3). Finally, we take away the toric part from $G$ by one last application of Ono's Lemma (step 4).

\paragraph*{Step 0: We may assume $G^\ssu$ to be simply connected.} Indeed, by Lemma \ref{lemme reduction sc}, we know that we may view $V$ as a homogeneous space of a bigger connected group $G'$ such that $(G')^\ssu$ is simply connected. We may thus assume that $G^\ssu$ is simply connected from now on. Note that $H^\f$ is not modified by this assumption.

\paragraph*{Step 1: Reduction to $G_1/H_1$ with $H_1=H^\torf$ and $G_1=\sl_{n_1}\times G^\tor$.} Consider the quotient $H^\torf$ and embed it into $\sl_{n_1}$ for some $n_1$. Denote by $G'_1$ the direct product $\sl_{n_1}\times G$. Composing with the projection $H\to H^\torf$ we get a morphism $H\to \sl_{n_1}$ and hence a diagonal inclusion $H\hookrightarrow G'_1$. Denote by $W_1$ the homogeneous space $G'_1/H$ thus obtained. Since $\sl_{n_1}$ is special and $k$-rational, Lemma \ref{lemme stable bir} tells us that $W_1$ is $k$-stably birational to $V$, which means that we are reduced to prove the theorem for this homogeneous space.\\

Apply then Lemma \ref{lemme fibration} to $G'_1$ an to its normal subgroup $G^\ssu\subset G\subset G'_1$. We get a homogeneous space $V_1$ of $G_1:=\sl_{n_1}\times G^\tor$ whose stabilizer $H_1$ is easily seen to be isomorphic to $H^\torf$, as well as a $G'_1$-equivariant map $W_1\to V_1$ whose fibers are homogeneous spaces of $G^\ssu$ with geometric stabilizer isomorphic to $H^\ssu$. Assume then that BM-WA holds for $V_1$ and let $(y_v)\in W_1(k_\Omega)^\br$ be a family of local points orthogonal to $\br W_1$ and $(z_v)$ be its image in $V_1$. By Proposition \ref{prop BM}(i), we have $(z_v)\in V_1(k_\Omega)^\br$ and hence, by our assumption, we know that we can find a $k$-point $x\in V_1(k)$ as close as we want to $(z_v)$. In particular, we may assume that $x_v$ is close to $z_v$ for every archimedean place of $k$. The fiber above $x$ is then a homogeneous space of the simply connected group $G^\ssu$ with connected stabilizer isomorphic to $H^\ssu$ which has a $k_v$-point for every archimedean place of $k$ since it is close enough to the respective $y_v$'s (here we use the fact that $W_1\to V_1$ is smooth and hence $W_1(k_v)\to V_1(k_v)$ is an open map). By \cite[Prop. 3.4]{Borovoi96}, we know then that this fiber verifies weak approximation, which tells us that we may approach $(y_v)$ as much as we want by a $k$-point. This reduces the proof to the homogeneous space $V_1=G_1/H_1$.

\paragraph*{Step 2: reduction to $G_2/H_2$ where $G_2=\sl_{n_2}\times G^\tor$ and $H_2$ is an extension of $H^\f$ by a torus which is $H^\f$-isogenous to an $H^\f$-quasi-split torus.}
We have now $V_1=G_1/H_1$ with $G_1=\sl_{n_1}\times G^\tor$ and $H_1=H^\torf$. In order to ease notation, let us note $T$ the torus $H_1^\tor=H^\tor$ during this step.\\

Note first that $T$ has a natural structure of $H^\f$-torus. By Corollary \ref{lemme super Ono}, there exists an integer $r\geq 1$ and quasi-split $H^\f$-tori $P, Q$ such that there is an $H^\f$-isogeny $T^r\times P\to Q$ (see definitions in section \ref{section Ono}). Define \begin{align*}
R &:=T^{r-1}\times P, \\
S &:=T^r\times P=T\times R, \\
H_2 &:= R\rtimes H_1,
\end{align*}
where $H_1$ acts diagonally on $R$ via its quotient $H^\f$. It is easy to see that $H_2$ is an extension of $H^\f$ by $S$ inducing the natural action of $H^\f$ on each copy of $T$ and on $P$.\\

Recall now that the group $H_1=H^\torf$ was embedded by construction into $\sl_{n_1}$ in step 1. Up to taking a bigger $n_1$ at that point, we may assume then that the whole group $H_2$ embeds into $\sl_{n_1}$ and that the inclusion of $H_1$ in $G_1$ factors through this embedding. Recall also that $G_1$ is a direct product of a torus and $\sl_{n_1}$ and hence it acts on $\sl_{n_1}$ by conjugation. We may thus consider a new copy of $\sl_{n_1}$, which we denote by $\sl_{n_1}'$, and the semi-direct product
\[G_2':=\sl_{n_1}'\rtimes G_1=(\sl_{n_1}'\rtimes\sl_{n_1})\times G^\tor.\]
The subgroup $H_1\subset G_1$ defines a subgroup $\sl_{n_1}'\rtimes H_1\subset G_2'$. Now, we know that there is a copy $H_2'$ of $H_2$ lying inside this normal subgroup $\sl_{n_1}'$ and hence a copy $R'$ of $R$ inside it and, since the action is by conjugation, this copy of $R$ will clearly be $H_1$-invariant (recall that $H_2=R\rtimes H_1$). We may thus consider the subgroup $R'\rtimes H_1$ of $G_2'=\sl_{n_1}'\rtimes G_1$, where $R'\subset \sl_{n_1}'$ and $H_1\subset G_1$. This subgroup is isomorphic to $H_2$ and its image under the projection $G_2'\to G_1$ is $H_1$. To summarize, we have
\[\begin{array}{ccccc}
G_2' & = &\sl_{n_1}' &\rtimes & G_1 \\
\cup && \cup && \cup \\
H_2 & \cong & R' &\rtimes & \,H_1.
\end{array}\]
In particular, by Lemma \ref{lemme fibration} applied to the normal subgroup $\sl_{n_1}'$, we have a projection $W_2:=G_2'/H_2\twoheadrightarrow G_1/H_1=V_1$ for which we claim that there is a section $V_1\to W_2$. Indeed, as summarized by the diagram below, the natural section $G_1\to G_2'$ is clearly $H_1$-equivariant (here we view $H_1$ as a subgroup of $H_2$ via its natural section and the action is by multiplication on the right) and hence induces a section $G_1/H_1\to G_2'/H_1$ of the natural projection $G_2'/H_1\twoheadrightarrow G_1/H_1$. Now, it is clear that this last projection factors through $G_2'/H_2$, giving thus the section $G_1/H_1\to G_2'/H_2$ by composition.
\[\xymatrix{
G_2' \ar@{->>}[d] \ar@{->>}[r] & G_2'/H_1 \ar@{->>}[d] \ar@{->>}[r] & G_2'/H_2 \ar@{->>}[dl] \\
G_1 \ar@/_1pc/[u] \ar@{->>}[r] & \,G_1/H_1. \ar@/_1pc/[u]
}\]

Assume now that BM-WA holds for $W_2$ and let $(x_v)\in V_1(k_\Omega)^\br$. Then its image $(y_v)\in W_2(k_\Omega)$ via the section $V_1\to W_2$ is in $W_2(k_\Omega)^\br$ by Proposition \ref{prop BM}(ii). We can thus approximate as much as we want $(y_v)$ by a $k$-point in $W_2$. It suffices then to push down this $k$-point to $V_1$ in order to approximate $(x_v)$. Thus we are reduced to the case of $W_2=G_2'/H_2$.

Finally, embed now $H_2$ into some new $\sl_{n_2}$ for some $n_2$. Since $\sl_{n_1}\rtimes\sl_{n_1}$ is special and $k$-rational, Lemma \ref{lemme stable bir} allows us to replace our $G_2'=(\sl_{n_1}\rtimes\sl_{n_1})\times G^\tor$ by $G_2:=\sl_{n_2}\times G^\tor$ (and hence $W_2$ by $V_2:=G_2/H_2$).

\paragraph*{Step 3: reduction to $G_3/H_3$ where $G_3=G_2$ and $H_3$ is an extension of $H^\f$ by a finite abelian group.}
We have now $V_2=G_2/H_2$ with $G_2=\sl_{n_2}\times G^\tor$ and $H_2$ an extension of $H^\f$ by a torus $S$ which is $H^\f$-isogenous to the $H^\f$-quasi-split torus $Q$.\\

First of all, Proposition \ref{proposition existence} tells us that there exists a finite subgroup $H_3'\subset H_2$ surjecting onto $H^\f$, with $S\cap H_3'=S[m]$ for some $m\in\n$. On the other hand, the kernel of the $H^\f$-isogeny $S\to Q$ is by definition a finite group on which $H^\f$ acts. Denote then by $A$ the kernel of the composition
\[S\xrightarrow{m} S \to Q,\]
where $m$ denotes multiplication by $m$ on $S$. This is still a finite subgroup over which $H^\f$ acts since multiplication by $m$ is $H^\f$-equivariant. Define then $H_3$ to be the subgroup of $H_2$ generated by $H_3'$ and $A$. Since $A$ is $H^\f$-invariant and contains the $m$-torsion of $S$, it is easy to see that $H_3$ is a finite group, extension of $H^\f$ by $A$. Putting $G_3:=G_2$, we have then the following commutative diagram
\[\xymatrix@=2mm{
& G_2 \ar[dd] \ar[dr] & \\
& & G_3/H_3 \ar@{=}[r] \ar[dl] & V_3 \\
V_2 \ar@{=}[r]& G_2/H_2.
}\]
The vertical arrow is a right $H_2$-torsor, defining thus en element $[X]$ of $H^1(k(V_2),H_2)$ when one takes the generic fiber. A classic result in non-abelian cohomology (cf. \cite[I.5.3]{SerreCohGal}) tells us then that $X$ is also a left $\asd{}{X}{}{2}{H}$-torsor. It is not difficult then to see that the generic fiber of the lower right arrow is a left $k(V_2)$-homogeneous space $W_3$ of the group $\asd{}{X}{}{2}{H}$ whose geometric stabilizer is $H_3\times_k \overline{k(V_2)}$. In other words, we have the diagram
\[\xymatrix@=2mm{
X \ar[dd] \ar[dr] & \\
& X/H_3 \ar[dl] \ar@{=}[r] & W_3  \\
k(V_2).
}\]
Now, since the geometric stabilizer is $H_3\times_k \overline{k(V_2)}$ and $H_3$ intersects every connected component of $H_2$, it is easy to see that the left action of $\asd{}{X}{\tor}{2}{H}=\asd{}{X}{}{}{S}$ on $W_3$ is transitive at the level of $\overline{k(V_2)}$-points, which means that $W_3$ is also a left homogeneous space of $\asd{}{X}{}{}{S}$ with geometric stabilizer $A\times_{k}\overline{k(V_2)}$. Now, this stabilizer is actually defined over $k(V_2)$ and is nothing but $\asd{}{X}{}{}{A}$ (recall that $H^\f$ acts on $A$ and that twisting is defined via this action), telling us finally that $W_3$ is a left \emph{torsor} of the \emph{quasi-split} torus
\[\asd{}{X}{}{}{Q}\cong\asd{}{X}{}{}{S}/\asd{}{X}{}{}{A},\]
quasi-splitness being a consequence of Proposition \ref{prop tore tordu} since $Q$ was $H^\f$-quasi-split. Then, by Shapiro's Lemma and Hilbert's Theorem 90, $W_3$ must be the trivial torsor. This tells us that $V_3$ is $k$-birational to the $k(V_2)$-torus $\asd{}{X}{}{}{Q}$ and, since this last is $k(V_2)$-rational by quasi-splitness, we see that $V_3$ is rational over $k(V_2)$, i.e. $k$-stably birational to $V_2$. This reduces us to the case of $V_3=G_3/H_3$.

\paragraph*{Step 4: Reduction to $G_4/H_4$ with $G_4=\sl_{n_4}$ and $H_4$ an extension of $H^\f$ by a finite abelian group.}
Let us note $G^\tor$ by $T'$ during this last step in order to ease notation. We have now $V_3=G_3/H_3$ with $G_3=\sl_{n_3}\times T'$ ($n_3=n_2$) and $H_3$ an extension of $H^\f$ by a finite abelian group $A$.\\

By Ono's Lemma (Lemma \ref{lemme Ono}), there exist quasi-split $k$-tori $P'$, $Q'$ such that there exists a $k$-isogeny $Q'\to {T'}^m\times P'$. Denote by $A'$ the corresponding finite abelian kernel and consider then the product
\[G_3\times {T'}^{m-1}\times P'=\sl_{n_3}\times {T'}^m\times P',\]
and its subgroup $H_3$ seen as a subgroup of the first factor. We clearly have then
\[W_4:=(G_3\times {T'}^{m-1}\times P')/H_3=V_3\times {T'}^{m-1}\times P'\]
Put $G_4':=\sl_{n_3}\times Q'$. We have then an isogeny with kernel $A'$
\[G_4'\to G_3\times {T'}^{m-1}\times P',\]
and hence $W_4$ is seen to be equal to $G_4'/H_4$, where $H_4$ is an extension of $H_3$ by $A'$. Note moreover that $H_4$ is still an extension of $H^\f$ by an abelian group. Indeed, the kernel of the natural projection $H_4\to H^\f$ is an extension of $A$ by $A'$, contained in $G_4'$, whose projections to both $\sl_{n_3}$ and $Q'$ are clearly abelian, hence the whole extension is abelian too.\\

Assume now that BM-WA holds for $W_4$. Since clearly the projection $W_4\to V_3$ has a section, we know by the same argument given at the end of step 2 that we may reduce us to the case of $W_4$.

Embed then $H_4$ into some new group $G_4:=\sl_{n_4}$ for some $n_4$. Then both $G_4$ and $G_4'$ are special $k$-rational groups since $Q'$ is quasi-split and hence $k$-rational. A double application of Lemma \ref{lemme stable bir} tells us then that $G_4/H_4$ is $k$-stably birational to $G'_4/H_4$, so that $H_4$ is the finite group $F$ embedded into $\sln$ that we were looking for.
\end{proof}

\section{An unconditional result}\label{section resultat inconditionnel}
In \cite[Thm. 1]{HarariBulletinSMF}, Harari used fibration methods to prove that if BM-WA holds (see notations in section \ref{section BM}) for a quotient $V=\sln/G$ with $G$ a finite $k$-group, then it still holds for a quotient $\sln/E$ with $E=A\rtimes G$ and $A$ a finite abelian $k$-group (recall that all these properties are independent of the embeddings into $\sln$).

As a particular case of Theorem \ref{MainThm}, we may now add a connected component to this stabilizer. Indeed, we have:

\begin{pro}\label{prop david++}
Let $V=G/H$ be a homogeneous space such that $G$ is simply connected and $H^\torf$ is a semi-direct product of $H^\f$ and $H^\tor$ (see notations in section \ref{section HS}). Assume that $\sln/H^\f$ satisfies BM-WA for some embedding of $H^\f$ in $\sln$. Then $V$ satisfies BM-WA too.
\end{pro}

Recall that $G$ is said to be simply connected if $\bar k[G]^*/\bar k^*=1$ and $\pic \bar G=0$. Equivalently, this amounts to $G^\torf=1$ and $G^\ss$ being simply connected.

\begin{proof}
It suffices to follow the proof of the Main Theorem in this particular case:
\begin{itemize}
 \item There is no use for step 0 since $G^\ss$ is already simply connected.
 \item Step 1 changes the stabilizer $H$ by $H_1=H^\torf$ and hence we still have a semi-direct product of $H^\f$ and a torus.
 \item In step 2, one constructs a bigger stabilizer $H_2$ which is a semi-direct product of $H_1$ and some torus, hence $H_2=H_2^\torf$ will still be a semi-direct product of  $H^\f$ and a torus.
 \item The finite group $H_3$ constructed in step 3 is easily seen to be the semi-direct product of $H^\f$ and the kernel of the isogeny from $H_2^\torf$ to some quasi-split torus (note that Proposition \ref{proposition existence} is not even needed, i.e. one can take $m=1$ here).
 \item Finally, there is no use for step 4 either since $G^\tor$ is trivial and hence $G_3$ is already isomorphic to $\sln$. Hence $F=H_3$.
\end{itemize}
By Harari's result cited above, the homogeneous space $\sln/F$ satisfies BM-WA if $\sln/H^\f$ does, which concludes the proof.
\end{proof}

In particular, applying the few cases where we know that BM-WA holds (cf. \cite[\S 1.4]{HarariBulletinSMF}, \cite[Thm. 1]{GLA AF PH res}), we get:

\begin{cor}
Let $V=G/H$ be a homogeneous space such that $G$ is simply connected and $H^\torf$ is a semi-direct product of $H^\f$ and $H^\tor$. Assume that one of the following holds:
\begin{enumerate}
\item $H^\f$ is obtained by successive semi-direct products of abelian groups;
\item $H^\f$ is split by an extension $L/k$ such that the order of $\mu(L)$ is prime to the order of $H^\f$ (in particular, the order of $H^\f$ is odd).
\end{enumerate} 
Then $V$ satisfies BM-WA.
\end{cor}

Let us recall that condition 1 on $H^\f$ means that there exist a series of \emph{split} exact sequences
\[\xymatrix{
1 \ar[r] & A_i \ar[r] & H_i \ar[r] & H_{i-1} \ar[r] \ar@/_1pc/[l] & 1,
}\]
for $1\leq i\leq n$, where $H_n=H^\f$, $H_0=\{1\}$ and the $A_i$'s are abelian $k$-groups.

\begin{proof}
This is a direct consequence of Proposition \ref{prop david++} once one remarks that, for such an $H^\f$, we have BM-WA for $\sln/H^\f$ under condition 1 (cf. \cite[\S 1.4]{HarariBulletinSMF}) and weak approximation for $\sln/H^\f$ under condition 2 (cf. \cite[Thm. 1]{GLA AF PH res}).
\end{proof}

\end{document}